\documentclass[11pt]{article}
\usepackage{amssymb,amsmath,amsthm}
\usepackage[colorlinks=true,urlcolor=blue,
citecolor=red,linkcolor=blue,linktocpage,pdfpagelabels,bookmarksnumbered,bookmarksopen]{hyperref}
\usepackage{graphicx}
\usepackage[english]{babel}
\usepackage{epsfig}
\usepackage{graphics}
\newtheorem{lemma}{Lemma}[section]
\newtheorem{theorem}[lemma]{Theorem}
\newtheorem{corollary}[lemma]{Corollary}
\newtheorem{rem}[lemma]{Remark}

\theoremstyle{definition}

\newtheorem{definition}[lemma]{Definition}

\def\div{\mathop{\mathrm{div}}}

\newcommand{\N}{{\mathbb N}}
\newcommand{\R}{{\mathbb R}}

\newcommand{\Om}{\Omega}

\newcommand{\rg}{\rightarrow}

\def\bean#1\eean{\begin{eqnarray*}#1\end{eqnarray*}}

\begin{document}

\title{On the pure critical exponent problem for the $p$-Laplacian}

\author{ Carlo Mercuri%\\ Department of Mathematics and Computer Science
%\\
%Technische Universiteit Eindhoven
%\\
%Postbus 513,
%5600 MB Eindhoven
%\\
%The Netherlands
\thanks{ Institute for Complex Molecular Systems and Department of Mathematics and Computer Science,
Technische Universiteit Eindhoven,
Postbus 513,
5600 MB Eindhoven,
The Netherlands.  E-mail: c.mercuri@tue.nl} \and Filomena Pacella %\\ Dipartimento di Matematica Guido Castelnuovo
%\\
%Universit\`a di Roma La Sapienza
%\\ P.le A. Moro 2
%00185 Rome\\Italy
\thanks{ Dipartimento di Matematica Guido Castelnuovo
,
Universit\`a di Roma "Sapienza"
, P.le A. Moro 2
00185 Rome, Italy. E-mail: pacella@mat.uniroma1.it} }

\date{}

\maketitle

\begin{abstract}
 In this paper we prove existence and multiplicity of positive and sign-changing solutions to the pure critical exponent problem for the $p$-Laplacian operator with Dirichlet boundary conditions on a bounded domain having nontrivial topology and discrete symmetry. Pioneering works related to the case $p=2$ are H. Brezis and L. Nirenberg \cite{BN}, J.-M. Coron \cite{coron}, and A. Bahri and J.-M. Coron \cite{BCoron}. A global compactness analysis is given for the Palais-Smale sequences in the presence of symmetries.
\\

{\sc Keywords:} Critical Sobolev exponent, $p$-Laplacian, Palais-Smale sequences, lack of compactness, sign-changing solutions. \\
\\
{\sc Mathematics Subject Classification (2010): 35J20, 35J66, 35J92}
\end{abstract}

%--------------------------------------------------------------------------

\section{Introduction}
We tackle the following problem with pure critical nonlinearity

\begin{equation}\label{main eq}
\left\{
\begin{array}{lll}
-\Delta_p u = |u|^{p^*-2}u  \quad &\mathrm{in}& \,\,\Omega   \\

  u=0 &\mathrm{on}& \,\, \partial \Omega,
\end{array}
\right.
\end{equation}
where  $\Omega$ is a smooth bounded domain in $\R^N,$ $1<p<N,$ $p^*:=N p/(N-p)$ is the critical Sobolev exponent, $\Delta_p u:=\div(|\nabla u|^{p-2}\nabla u)$ is the $p$-Laplace operator defined on $$\mathcal D^{1,p}(\mathbb R^N):=\{u\in L^{p^*}(\R^N):\nabla u\in L^p(\mathbb R^N;\mathbb R^N)\}$$ endowed with the norm $$||u|| :=||\nabla u||_{L^p(\mathbb R^N)}.$$\\
We denote by $ W_0^{1,p}(\Omega)$ the closure of $\mathcal D (\Omega)$ in $\mathcal D^{1,p}(\R^N)$ and define on $W_0^{1,p}(\Omega)$ the functional

\begin{equation}\label{J}
J(u):=\frac{1}{p}\int_{\Omega}|\nabla u|^p dx-\frac{1}{p^*} \int_{\Omega}|u|^{p^*}dx.
\end{equation}
We recall the definition of Nehari manifold

\begin{equation*}
\mathcal N(\Omega):=\{u\in W^{1,p}_0(\Omega)\setminus\{0\}\,:\, (J'(u),u)=0\}.
\end{equation*}
We also define the level
\begin{equation}\label{best}
c_\infty:=\inf\{J(u),\, u\in \mathcal N (\Omega)\}=\frac{S^{N/p}}{N},
\end{equation}
where
\begin{equation}\label{bbest}
S:=\inf\{\int_{\R^N}|\nabla u|^p dx,\, u\in W^{1,p}(\R^N)\, : \, \int_{\R^N}|u|^{p^*} dx =1\}
\end{equation}
is the best Sobolev constant, attained by nowhere zero (well known) functions in $\R^N,$ see e.g.\cite{T}.
It is well known that the infimum in (\ref{best}) does not depend on the domain. \\ Since the embedding of $W^{1,p}_0(\Omega)$ into $L^{p^*}(\Omega)$ is not compact, the functional $J$ does not satisfy the classical Palais-Smale condition  and this, in turns, does not allow to solve (\ref{main eq}) by standard variational methods. \\
In the case $p=2$ it is well known that the existence of solutions depends on the domain. Pohozaev's identity \cite{poh} together with the unique continuation principle (see e.g. \cite{Hormander}) implies that problem (\ref{main eq}) does not have a nontrivial solution (neither positive, nor sign-changing) if $\Omega$ is strictly starshaped. On the other hand, if the domain is an annulus $A$ the existence of a radial positive solution is provided, for any $p\in (1,N),$ by direct minimization methods, as a consenquence of the compactness of the embedding $W^{1,p}_{0, rad}(A) \hookrightarrow L^{p^*}(A).$\\
These two positive and negative results motivated the study of (\ref{main eq}) in topologically nontrivial domains in order to get existence of solutions. In this direction, in the case $p=2$,  the two main results are the one of Coron \cite{coron} for annular shaped domains with a small hole and the one of Bahri-Coron \cite{bahricoron} for domains with nontrivial homology.\\
However results of \cite{Dancer}, \cite{ding} and \cite{Pass} show that hypotheses of nontrivial topology are not necessary to get solutions, in the case $p=2.$ So the question of characterizing the domains for which a solution of (\ref{main eq}) exists is still open, even in the case $p=2.$\\
Concerning multiplicity of positive and sign-changing solutions several results have been obtained (case $p=2$) for domains which are small perturbations of a given domain $D,$ in particular if $D$ is a domain with a small hole, see e.g. \cite{ClappGrossiPistoia,clappmussopistoia,GrossiPacella,Li,MarchiPacella,MussoPistoia}. But for domains which are not such perturbations multiplicity remains largely open. A first result in this direction was recently established in \cite{clapppacella}, under some symmetry assumptions, always in the case $p=2.$ \\
Coming to the case $p\neq2$ let us start observing that a Pohozaev type nonexistence result is not yet available for sign-changing solutions of (\ref{main eq}), as the unique continuation principle for the $p$-Laplacian is not known, see e.g. \cite{mercuriwillem}, while for nonnegative solutions has been proved in \cite{GV}. The nonexistence of sign-changing radial solutions in the case $\Omega:=\{x\in\R^N\,:\, |x|<1\}$ holds in the range $2N/(N+2)\leq p\leq 2,$ as observed in \cite{mercuriwillem} by an ODE argument.\\
As far as existence of solutions is concerned there are no results at all, except for the above mentioned one concerning radial solutions in the annulus. This is due to serious difficulties arising in dealing with the quasilinear case. Let us outline the main ones.
A first obstacle is given by the fact that even the positive solutions of (\ref{main eq}) on the whole $\R^N$ are not classified, except for the optimizers of the Sobolev constant.
In the case $p=2$ they are explicitly computed (see \cite{GS}) and this is the core of many existence results, see e.g. \cite{bahricoron} and \cite{coron}.
Moreover the nonexistence of solutions in the half-space which holds for $p=2$ is another major problem in extending the compactness result of \cite{St}, (see \cite{mercuriwillem} and \cite{mercurisquass}), which is used extensively in the semilinear case.
Finally the Lyapunov-Schmidt reduction method, largely exploited in perturbation results when $p=2,$ does not seem applicable because the study of the linearized operator of the $p$-Laplacian is another major question (see e.g. \cite{AftPac,DaSc}).\\
So the natural question is whether some topological and/or geometrical assumptions on the domain could still lead to existence results as for the case $p=2,$ in spite of the above mentioned difficulties.\\
In the present paper we impose some symmetry on the domain which allows, with an extension of the recent global compactness theorem of \cite{mercuriwillem} to get, as in \cite{clapppacella}, existence and multiplicity of positive and sign-changing solutions.\\
We obtain two kind of results. The first one is a Coron type result, i.e. we consider domains with a small hole. In this case we impose little symmetry and get both existence and multiplicity of positive and sign-changing solutions. The second type of results is obtained for domains with a hole of any size but assuming more symmetry. In both cases we follow the same approach in \cite{clapppacella}. \\
So we assume that $\Omega$ is annular shaped, i.e. $0\notin \bar{\Omega}$ and contains $ A_{R_1,R_2}:=\{x\in\mathbb R^N\,:\,R_1<|x|<R_2\},$ and in addition, $\Omega$ is invariant under the action of a closed subgroup $G$ of orthogonal transformations of $ \mathbb R^N.$\\
Hereafter, given a set $A$ of functions and a subgroup $G$ of linear isometries of $\R^N,$ we will denote by $A^G$ the subset of $A$ given by $G$-symmetric functions. Furthermore, define the level
\begin{equation}\label{neh}
c(R_1,R_2):=\inf\{J(u)|\, u\in \mathcal N (A_{R_1,R_2})^{O(N)}\}.
\end{equation}
Finally, define the cardinality of the minimal $G$-orbit in $ \R^N\setminus \{0\}$ \begin{equation}\label{minimal}l=l(G)=\min \{\# G x\, :\, x\in \R^N\setminus \{0\}\}.\end{equation}
The following theorem provides a positive solution for a given symmetry closed subgroup of $O(N),$ by shrinking the size of the hole.
\begin{theorem}\label{coronn}
Let $1<p<N$ and $l(G)\geq2.$ Then, for every  $\delta>0$ there exists $R_\delta>0$ such that equation (\ref{main eq}) possesses a positive $G$-symmetric solution $u$ provided $R_1/R_2<R_\delta.$ Furthermore $$J(u)\leq\frac{S^{N/p}}{N}+\delta.$$
\end{theorem}
The existence of a positive solution for a fixed hole size is given by the following
\begin{theorem}\label{fixhol}
Let $1<p<N$ and $0<R_1<R_2$ given. Then, there exists a positive $l_0$ depending on $p$ and $R_1/R_2$ such that for every closed subgroup $G\subset O(N)$ with $l(G)>l_0$ equation  (\ref{main eq}) possesses a positive $G$-symmetric solution, with $$J(u)\leq c(R_1,R_2).$$
\end{theorem}
Moreover,  in the spirit of Theorem 1.2 and Theorem 1.3 of \cite{clapppacella}, the following theorems provide existence of multiple sign-changing solutions.
\begin{theorem}[Fixed symmetries-small hole]\label{coronn2}
Let $1<p<N$ and $l(G)\geq2.$ Then, for every  $\delta>0$ there exists $R_\delta>0$ such that equation (\ref{main eq}) possesses $l$ pairs of sign-changing $G$-symmetric solutions $\pm u_1,..., \pm u_l$ provided $R_1/R_2<R_\delta.$ Furthermore $$J(u_k)\leq(k+1)\frac{S^{N/p}}{N}+\delta, \quad k=1,...,l$$
\end{theorem}

\begin{theorem}[Fixed hole-more symmetries]\label{fixhol2}
Let $1<p<N,$ $m\in \mathbb N$ and $0<R_1<R_2$ given. Then, there exists a positive $l_0$ depending on $p,m$ and $R_1/R_2$ such that for every closed subgroup $G\subset O(N)$ with $l(G)>l_0$ equation  (\ref{main eq}) possesses $m$ pairs of sign-changing $G$-symmetric solutions $\pm u_1,..., \pm u_m,$  with $$J(u_k)\leq (k+1)c(R_1^{\frac{1}{m+1}},R_2^{\frac{1}{m+1}}), \quad k=1,...,m.$$
\end{theorem}
As far as we know the previous stated results are the first ones for the pure critical exponent problem (\ref{main eq}) for the $p$-Laplacian, $p\neq 2.$ \\
The paper is organized as follows. In Section \ref{GCsec} we state the global compactness results which we need in order to study the Palais-Smale condition for the functional $J,$ this extends the result of \cite{clapp,mercuriwillem}; in Section \ref{ene} we construct a family of positive solutions belonging to the Nehari manifold in a hierarchy of annular domains which is suitable in order to apply the minimax scheme used in \cite{clapppacella}; in sections \ref{proofpo} and \ref{proofsign} we prove the results stated above, while in Section \ref{proofglo}  we sketch the proof of the global compactness results stated in Section \ref{GCsec}; finally in Section \ref{ext} we give some comments and an extension to a more general symmetric setting.

\section{Global compactness results in the presence of symmetries}\label{GCsec}
The compactness results of the present section are meant as an extension of the results of \cite{clapp} and \cite{mercuriwillem}. \\
In the sequel $G$ will denote a closed subgroup of orthogonal transformations of $\R^N.$ The action of $G$ on a function $u: \R^N\rightarrow \R$ is defined through
$$
(gu)(x):=u(g^{-1}x),\,\,x\in\Omega, g\in G.
$$
In particular we consider $G$-invariant PS-sequences, i.e., PS sequences in the closed subspace of $W^{1,p}_0(\Omega)$ defined through
$$
W^{1,p}_0(\Omega)^G:=\{u\in W^{1,p}_0(\Omega)\,:\, gu=u\}.
$$
A domain of $\Omega\subset\R^N$ is $G$-invariant if $gx\in \Omega$ whenever $x\in \Omega,$ for all $g\in G.$ Finally we say that a function $u$ is $G$-invariant if $gu=u$ whenever $g\in G.$
We define  on $W^{1,p}_0(\Omega)$ $$\phi(u)=\int_{\Omega}\frac{|\nabla u|^{p}}{p}+a(x)\frac{|u|^p}{p}-\frac{|u|^{p^*}}{p^*}dx,$$ and on $\mathcal D^{1,p}(\R^N)$

 $$\phi_\infty(u)=\int_{\R^N}\frac{|\nabla u|^{p}}{p}-\frac{u_+^{p^*}}{p^*}dx.
 $$
In the following theorem we assume that  \newline
\\
\textbf{(A)} \quad \quad {\it  $\Omega$ is a smooth bounded domain of $\R^N$, $1<p<N, $ $a\in L^{N/p}(\Omega).$}
\\
\textbf{(B)} \quad \quad  {\it $\Omega, \, a$ are $G$-invariant.}
\\
Boundedness of PS sequences can be guaranteed by assuming that $a$ is such that
$$
  \mathop{\inf_{u \in W^{1,p}_0(\Omega)}}_{\|\nabla u\|_{L^p}=1} \int_{\Omega}|\nabla u|^{p}+a(x)|u|^pdx>0.
$$
See \cite{mercuriwillem}.

%Let us recall that $$<\phi'(u),h>= \int_{\Omega}[|\nabla u|^{p-2}\nabla u \cdot \nabla h+a|u|^{p-2}u \, h-|u|^{p^*-2} %u \,h]dx,$$

%$$<\phi_{\infty}'(u),h>= \int _{\R^N}[|\nabla u|^{p-2}\nabla u\cdot \nabla h-\mu u_+^{p^*-1}\,h]dx.$$
We recall, see e. g. \cite{clapp,diek}, that for a given $y\in \R^N,$ there exists a homeomorphism $Gy\simeq G/ G_y,$ where $Gy:=\{gy\in \R^N,\, g\in G\}$ is the $G$-orbit of $y$ and $G_y$ is the isotropy subgroup of $y,\, G_y:=\{g\in G , gy=y\}.$ The index of $G_y$ in $G,$ which we denote by $|G/G_y|,$ is therefore the cardinality of $Gy.$ \newline
We use the following notation:
$$
\begin{array}{l}
u_+=\max(u,0),u_-=\max(-u,0)\\
\\
\R^N_+=\{x\in\R^N:x_N>0\}.
\end{array}
$$
In view of getting a Coron-type result we have
\begin{theorem}[Palais-Smale sequences nearby the positive cone]\label{GlobalPositive}
Let $1<p<N.$ Under assumptions \textbf{(A)} and \textbf{(B),} let $\{u_n\}_n \subset W^{1,p}_{0} (\Omega)^G$ be a bounded sequence such that
 $$\phi(u_n)\rightarrow c \quad \quad \phi'(u_n)\rightarrow 0 \quad \textrm{in} \,\,W^{-1,p'}(\Omega) $$ and
 $$\|(u_n)_-\|_{L^{p^*}(\Omega)}\rightarrow 0, \quad n\rightarrow \infty.
 $$
 Then, passing if necessary to a subsequence, there exists a solution $v_0\in W^{1,p}_{0}(\Omega)^G$ of
\begin{eqnarray*}
-\Delta_p u +a(x) u^{p-1}= u^{p^*-1} &\textrm{in}& \Omega, \\
u \geq 0 & \textrm{in} & \Omega,
\end{eqnarray*}
a finite number $k\in \mathbb N$ of closed subgroups $\Gamma_1,...,\Gamma_k$ of finite index in $G$ and, correspondingly, $\{v_1,...,v_k\}\subset \mathcal D^{1,p}(\R^N)$ $\Gamma_i-$invariant solutions of
\begin{eqnarray*}
-\Delta_p u = u^{p^*-1} &\textrm{in}& \R^N, \\
u \geq 0 & \textrm{in} & \R^N,
%u\rightarrow 0 & \textrm{as} & |x|\rightarrow\infty,
\end{eqnarray*}
 $k$ sequences $\{y^i_n\}_n \subset \Omega$ and $\{\lambda^i_n\}_n \subset \R_+,$  satisfying
$$G_{y^i_n}=\Gamma_i,\forall n, \qquad y^i_n\rightarrow y^i\in \Omega, \, \textrm{as}\,\, n\rightarrow \infty,$$

$$\frac{1}{\lambda^i_n}\, \textrm{dist} \, (y^i_n,\partial\Omega)\rightarrow \infty , \,\quad n\rightarrow \infty, $$

$$
\frac{1}{\lambda^i_n}|gy^i_n-g'y^i_n|\rightarrow \infty , \,\quad n\rightarrow \infty,\, \forall [g]\neq[g']\in G/ \Gamma_i
$$

$$\|u_n-v_0-\sum^k_{i=1}\sum_{[g]\in G/ \Gamma_i}(\lambda^i_n)^{(p-N)/p}v_i (g^{-1}(\cdot-gy^i_n)/\lambda^i_n)\|\rightarrow 0, \quad n\rightarrow \infty,$$

$$\|u_n\|^p\rightarrow \|v_0\|^p+\sum^k_{i=1}|G/ \Gamma_i| \|v_i\|^p, \quad n\rightarrow \infty,$$

$$\phi(v_0)+\sum^k_{i=1}|G/ \Gamma_i| \phi_\infty (v_i)=c.$$

{\it For $i\geq 1$ there holds $v_i>0,$ whenever $v_i$ is non-trivial, by the strong maximum principle \cite{Vaz}, see also [Theorem 1.2, \cite{mercuriwillem}].}\newline
\end{theorem}
The proof will be given in Section \ref{proofglo}.
For sign-changing solutions we will need the following
\begin{theorem}[Sign-changing Palais-Smale sequences]\label{GlobalSign}
Let $1<p<N,$ $\phi$ be as above, and define $\phi_\infty$ in a slightly different way:
 $$\phi_\infty(u):=\int_{\R^N}\frac{|\nabla u|^{p}}{p}-\frac{|u|^{p^*}}{p^*}dx.
 $$
 Under assumptions \textbf{(A)} and \textbf{(B),} let $\{u_n\}_n \subset W^{1,p}_{0} (\Omega)^G$ be a bounded sequence such that
 $$\phi(u_n)\rightarrow c \quad \quad \phi'(u_n)\rightarrow 0 \quad \textrm{in} \,\,W^{-1,p'}(\Omega). $$

 Then, passing if necessary to a subsequence, there exists a solution $v_0\in W^{1,p}_{0}(\Omega)^G$ of
$$
-\Delta_p u +a(x) |u|^{p-2}u=  |u|^{p^*-2}u \quad\textrm{in} \,\, \Omega,
$$
a finite number $k\in \mathbb N$ of closed subgroups $\Gamma_1,...,\Gamma_k$ of finite index in $G$ and, correspondingly, $\{v_1,...,v_k\}\subset W_0^{1,p}(H_i)$  $\Gamma_i-$invariant solutions of
$$
-\Delta_p u = |u|^{p^*-2}u \quad\textrm{in}\,\, H_i
$$
where $H_i$ is either $\R^N$ or (up to translation and rotation) $\R_+^N$ and there exist $k$ sequences $\{y^i_n\}_n \subset \Omega$ and $\{\lambda^i_n\}_n \subset \R_+,$  $\lambda^i_n\rightarrow 0^+,$ satisfying

$$G_{y^i_n}=\Gamma_i,\forall n, \qquad y^i_n\rightarrow y^i\in \Omega\cup\partial \Omega, \, \textrm{as}\,\, n\rightarrow \infty,$$

$$\frac{1}{\lambda^i_n}\, \textrm{dist} \, (y^i_n,\partial\Omega)\rightarrow \infty , \,\quad n\rightarrow \infty $$
if $H_i=\R^N,$ or
$$\frac{1}{\lambda^i_n}\, \textrm{dist} \, (y^i_n,\partial\Omega)< \infty , \,\quad n\rightarrow \infty $$
if (up to translation and rotation) $H_i=\R_+^N,$ and
$$
\frac{1}{\lambda^i_n}|gy^i_n-g'y^i_n|\rightarrow \infty , \,\quad n\rightarrow \infty,\, \forall [g]\neq[g']\in G/ \Gamma_i
$$

$$\|u_n-v_0-\sum^k_{i=1}\sum_{[g]\in G/ \Gamma_i}(\lambda^i_n)^{(p-N)/p}v_i (g^{-1}(\cdot-gy^i_n)/\lambda^i_n)\|\rightarrow 0, \quad n\rightarrow \infty,$$

$$\|u_n\|^p\rightarrow \|v_0\|^p+\sum^k_{i=1}|G/ \Gamma_i| \|v_i\|^p, \quad n\rightarrow \infty,$$

$$\phi(v_0)+\sum^k_{i=1}|G/ \Gamma_i| \phi_\infty (v_i)=c.$$
\end{theorem}
As for the previous theorem, the proof will be given in Section \ref{proofglo}.\\
Now define the functional $J$ as in (\ref{J}).
\begin{definition}
{\it Given a subset $A\subset W^{1,p}_{0}(\Omega)^G,$ we say that the functional $J$ (see (\ref{J})) satisfies $(PS)_c$ relative to $A$ in $W^{1,p}_{0}(\Omega)^G$ if every sequence $\{u_n\}_n\subset W^{1,p}_{0} (\Omega)^G$ such that
$$u_n \notin A, \quad \quad J(u_n)\rightarrow c \quad \quad J'(u_n)\rightarrow 0 \quad \textrm{in} \,\,W^{-1,p'}(\Omega),$$
contains a strongly convergent subsequence. If $A=\oslash$ we say that $J$ satisfies $(PS)_c$ in $W^{1,p}_{0} (\Omega)^G.$}
\end{definition}
Let $\mathcal P$ be the positive cone in $W^{1,p}_{0} (\Omega)^G,$ namely
$\mathcal P:= \{ u \in  W^{1,p}_{0} (\Omega)^G \, : \, u_- \equiv 0\},$
and define $J^d:=\{u\in W^{1,p}_{0} (\Omega)^G\, :\, J(u)\leq d\}$ and $$\Pi_d:= \mathcal P\cup (-\mathcal P)\cup J^d.$$
We will need the following
\begin{corollary}[Palais-Smale condition for $J$]\label{PSJc}
Let $c_\infty$ be defined by (\ref{best}) and $l$ as in (\ref{minimal}).
\begin{itemize}
\item [i)] $J$ satisfies $(PS)_c$ in $W^{1,p}_{0}(\Omega)^G$ for every $c<lc_\infty.$
\item [ii)] If $l\geq2$ there exists $\varepsilon_0$ such that $J$ satisfies $(PS)_c$ relative to $\Pi_0$ in $W^{1,p}_{0}(\Omega)^G$ for every $c<(l+1)c_\infty+\varepsilon_0.$
\end{itemize}

\end{corollary}
 \begin{proof}
In Theorem \ref{GlobalSign}, multiplying the equation
$$
-\Delta_p u = |u|^{p^*-2}u \quad\textrm{in}\,\, H_i
$$

by $(v_i)_+$ and $(v_i)_-,$ and using the Sobolev inequality, for $i=0,1,...,k$ then we have necessarily

$$
\phi_\infty(v_i)=0,
$$

$$
\phi_\infty(v_i)\geq c_\infty,
$$

$$
\phi_\infty(v_i)\geq 2 c_\infty,
$$
according to the cases $v_i$ is (respectively) zero, has constant sign or changes sign. Here we put $H_0:=\R^N$ and we extend $v_0$ by zero in $\R^N\setminus \Omega.$\\
In order to prove $(i)$ notice that necessarily we have $k<1,$ hence $k=0.$ \\
For $(ii)$ we argue as follows. Notice that for $l\geq2$ we have $(l+1)c_\infty+\varepsilon_0 \leq 2lc_\infty,$ for every $\varepsilon_0\in (0,c_\infty]. $ In particular let $\varepsilon_0\in (0,c_\infty]$ be such that $J$ has no sign-changing critical points $u\in W^{1,p}_{0}(\Omega)$ with $J(u)<c_\infty+\varepsilon_0.$ Pick $\{u_n\}_n$ such that $$u_n \notin \Pi_{0}, \quad \quad J(u_n)\rightarrow c \quad \quad J'(u_n)\rightarrow 0 \quad \textrm{in} \,\,W^{-1,p'}(\Omega),$$ and assume by contradiction that $u_n$ does not have a strongly convergent subsequence. Then by Theorem \ref{GlobalSign}
there exists a finite set $\{v_1,...,v_k\}\subset W_0^{1,p}(H_i)$ of $\Gamma_i-$invariant nontrivial solutions of
$$
-\Delta_p u = |u|^{p^*-2}u \quad\textrm{in}\,\, H_i,
$$
satisfying
$$\sum^k_{i=1}|G/ \Gamma_i| \phi_\infty (v_i)=c$$

$$\|u_n-\sum^k_{i=1}\sum_{[g]\in G/ \Gamma_i}(\lambda^i_n)^{(p-N)/p}v_i (g^{-1}(\cdot-gy^i_n)/\lambda^i_n)\|\rightarrow 0, \quad n\rightarrow \infty,$$

for some sequences of finite index subgroups $\Gamma_i,$ points $\{y^i_n\}_n \subset \Omega$ and scalars $\{\lambda^i_n\}_n \subset \R_+.$\\
Hence for every $c<(l+1)c_\infty+\varepsilon_0$ the only possibility is that $k=1$ and $v_1$ does not change sign. Hence either $v_1\in \mathcal P$ or $v_1\in -\mathcal P$ (notice that by the strong maximum principle \cite{Vaz} applied either to $v_1$ or (resp.) to $-v_1,$ then $v_1>0 $ or (resp.) $v_1<0$). But this is a contradiction since $u_n \notin \Pi_{0}.$

 \end{proof}
\begin{rem}
Theorem \ref{GlobalSign} provides a description of the lack of the Palais-Smale condition by means of possible nontrivial solutions of the aforementioned critical equation either in $\R^N$ or in the halfspace. The latter possibility we are not able to rule out for sign-changing PS sequences, except for the following case. Define as in \cite{clapp} the optimal constant $$\mu^G:= \min_{x\in\bar{\Omega}}(\# G x)\frac{S^\frac{N}{p}}{N}\leq \infty$$ and consider the particular case $c:=\mu^G$ and $a\equiv 0.$ Then, if $\{u_n\}_n\subset W^{1,p}_0(\Omega)^G$ does not converge strongly (in this case necessarily $\mu^G<\infty$), then, $\{v_i\}_{i=1,...,k}\equiv\{v_1\},$ with $v_1(x):=u(x)=[\alpha+\beta|x|^{p/(p-1)}]^{1-N/p},\quad \alpha,\beta>0$ being the optimizers of the Sobolev inequality, namely, the Talenti functions \cite{T}.
\end{rem}

\section{Energy calibration for solutions in annular domains}\label{ene}

We have the following
\begin{lemma}\label{homog}
There holds \begin{equation}\label{hom}
c(R_1,R_2)=c(R_1/R_2,1),
\end{equation}
where $c(R_1,R_2)$ is defined in (\ref{neh}). Moreover
\begin{equation}\label{cinf}
c(R,1)\rightarrow c_\infty,\quad \textrm{as} \,\,\,R\rightarrow 0.\end{equation}
\end{lemma}

\begin{proof}
After extending by zero outside the annulus, equation (\ref{hom}) follows by the Sobolev invariance by dilations $u_\lambda(\cdot):=\lambda^{(p-N)/p}u(\cdot/\lambda),$ and choosing in particular $\lambda=R_2.$    \\
In order to prove (\ref{cinf}) fix $\varepsilon>0$. Then there exist $R_\varepsilon>0$ and $\bar{u}_\varepsilon \in \mathcal N(B_{R_\varepsilon})^{O(N)}$ such that $$J(\bar{u}_\varepsilon)\leq c_\infty+\varepsilon,$$ for  $R_\varepsilon$ sufficiently large.
Since $\mathcal D (B_{R_\varepsilon}\setminus\{0\})^{O(N)}$ is dense in $W^{1,p}_0(B_{R_\varepsilon})^{O(N)},$ there exists $u'_\varepsilon \in \mathcal D (B_{R_\varepsilon}\setminus\{0\})^{O(N)}$ such that $$\frac{1}{N}\int_{B_\varepsilon}|\nabla u'_\varepsilon|^p dx\leq c_\infty+2\varepsilon.$$ Define by scaling $u_\varepsilon(\cdot):= R_\varepsilon^{\frac{N-p}{p}}u'_\varepsilon(R_\varepsilon \cdot).$ Notice that $\textrm {supp}\, u_\varepsilon \subset A_{\eta_\varepsilon,1}$ for some $\eta_\varepsilon>0$ small and $$c_\infty \leq c(\eta_\varepsilon, 1)\leq  \frac{1}{N}\int_{A_{\eta_\varepsilon,1}}|\nabla u_\varepsilon|^p dx=\frac{1}{N}\int_{B_\varepsilon}|\nabla u'_\varepsilon|^p dx\leq c_\infty+2\varepsilon.$$

\end{proof}

\begin{lemma}\label{cal}
Fix $0<R_1<R_2$ and $m\in \N.$ Then there exist radii $R_2=:r_0>r_1>...>r_m:=R_1,$ and nonnegative radial functions $\omega_i \in \mathcal N (A_{R_1,R_2}),$ $i=1,...,m$ such that
$$
\textrm{supp} \,\omega_i\subset A_{r_{i},r_{i-1}},\quad \textrm{and} \quad J(\omega_i)=c(R^{\frac{1}{m}}_1,R^{\frac{1}{m}}_2).
$$
Furthermore $\omega_i>0$ in $A_{r_{i},r_{i-1}},$ $\omega_i \in \mathcal C^{0, \alpha}(\R^N)$ after extending by zero outside $A_{r_{i},r_{i-1}},$ and $\mathcal C^{1, \alpha}(\overline{A_{r_{i},r_{i-1}}}).$
\end{lemma}
\begin{proof}
Define $m$ radii $R_2=:r_0>r_1>...>r_m:=R_1,$ through the condition $\frac{r_i}{r_{i-1}}= C, \forall i=1,...,m.$ It follows $C=(R_1/R_2)^{1/m}.$ By the Strauss inequality [Lemma 2.1, \cite{SWW}] the embedding
$\mathcal D^{1,p}_{rad}(\R^N)\subset L^{p^*}_{loc}(\R^N\setminus \{0\})$ is compact. Hence, by direct minimization, each $c(r_{i},r_{i-1})$ is attained by a function $\bar{\omega}_i \in W^{1,p}_{0,rad}(A_{r_{i},r_{i-1}}).$ Since $J(|u|)=J(u),$ $\bar{\omega}_i$ may be selected to be nonnegative.  The regularity follows from standard arguments, see e.g. [\cite{mercuriwillem}, p. 472]. The positivity is a consequence of the strong maximum principle \cite{Vaz}. Now extend $\bar{\omega}_i$ by zero outside $A_{r_{i},r_{i-1}},$ obtaining $\omega_i \in \mathcal N (A_{R_1,R_2}).$ By construction and from (\ref{hom}) we have $$J(\omega_{i+1})=c(r_{i+1},r_{i})=c(r_{i+1}/r_i,1)=c(C,1)=c({r_i}/r_{i-1},1)=c(r_i, r_{i-1})=J(\omega_i).$$ Finally from (\ref{hom}) we have $$J(\omega_i)=c((R_1/R_2)^{1/m},1)=c(R^{\frac{1}{m}}_1,R^{\frac{1}{m}}_2).$$

\end{proof}

\begin{rem}
In spite of [Remark 2.2, \cite{clapppacella}], we cannot conclude in our case that $$\tilde{\bar{\omega}}_m:=\sum^m_{i=1}(-1)^i\omega_i$$ is a $ C^{1, \alpha}$  sign-changing solution in $W^{1,p}_0(A_{R_1,R_2})^{O(N)}$.
\end{rem}

\section{Existence of a positive solution}\label{proofpo}
We prove Theorem \ref{coronn} in the spirit of Coron \cite{coron}. \\
\begin{proof}[Proof of Theorem \ref{coronn}]
Lemma \ref{homog} guarantees that for every  $\delta>0$ there exists $R_\delta>0$ such that $$c(R_1,R_2)\leq c_\infty+\delta,$$ provided $R_1/R_2<R_\delta.$ \newline
Notice also that the embedding $\mathcal N(A_{R_1,R_2})^{O(N)}\hookrightarrow \mathcal N(\Omega)^G  $ yields $$\inf_{\mathcal N(\Omega)^G}J(u)\leq c(R_1,R_2).$$
It follows that $$c:=\inf_{u\in W_0^{1,p}(\Omega)^G\setminus \{0\}}\max_{t>0}J(tu)=\inf_{\mathcal N(\Omega)^G}J(u)\leq c_\infty+\delta.$$ By the general minimax principle [see e.g. \cite{Will}], there exists a Palais-Smale sequence $\{u_n\}_n \subset W_0^{1,p}(\Omega)^G$ at level $c,$ such that, if for $\gamma_n\in X$  $$\max_{t\in
[0,1]}J (\gamma_n(t))\leq c+ \frac{1}{n},$$ then
\begin{equation}\label{positivecone}
\textrm{dist}(u_n, \gamma_n([0,1]))<\frac{1}{n},
\end{equation}
where \begin{align*}\label{Gammaeps}
X := \left\{\gamma \in C([0,1],W_0^{1,p}(\Omega)^G): \gamma(0)=0, J(\gamma(1))<0\right\}.
\end{align*}
Since $J (u)=J (|u|),$ it is possible to restrict on curves $\gamma\in C([0,1], \,\mathcal P).$
Defining
 % $$P:=\{u\in W_0^{1,p}(\Omega) \, :\, u_-\equiv 0\}, $$
 $$\mathcal P_{1/n}:=\{u\in W_0^{1,p}(\Omega)\, : \, \inf_{y\in \mathcal P}\|u-y\|_{W_0^{1,p}(\Omega)} < 1/n\}, $$
 equation (\ref{positivecone}) yields $u_n \in\mathcal P_{1/n}$ for large $n\in \mathbb N.$ By the Sobolev inequality, this implies  $\|(u_n)_-\|_{L^{p^*}(\Omega)}\rightarrow 0$, as $n\rightarrow \infty.$ See also [\cite{mercuriwillem}, p. 481]. \\
By Corollary \ref{PSJc} $i)$, using Theorem \ref{GlobalPositive} we can pick $\delta\in (0,c_\infty)$ such that $\{u_n\}_n$ is relatively compact. The restriction on $l\geq2$ allows us to have $c+\delta<l c_\infty.$ For such $\delta,$ the mountain-pass level $c$ is critical in $W_0^{1,p}(\Omega)^G,$ and by the symmetric criticality principle the duality can be extended to $W_0^{1,p}(\Omega).$ Hence, by construction, there exists a nontrivial nonnegative critical point of $J.$  The positivity follows from the strong maximum principle, and this concludes the proof.

\end{proof}
\begin{proof}[Proof of Theorem \ref{fixhol}]
The proof is the same as above, using Theorem \ref{GlobalPositive}. \\ For every closed subgroup $G\subset O(N)$ with $l=l(G)>l_0:=c_\infty^{-1}c(R_1,R_2),$ then $$c:=\inf_{u\in W_0^{1,p}(\Omega)^G\setminus \{0\}}\max_{t>0}J(tu)=\inf_{\mathcal N(\Omega)^G}J(u)\leq c(R_1,R_2)<l c_\infty.$$
By Corollary \ref{PSJc} $i)$, using Theorem \ref{GlobalPositive} we can pick a Palais-Smale sequence at level $c,$which is nearby the positive cone and relatively compact. Then $c$ is critical and there exists a positive $G$-symmetric solution $u,$ of (\ref{main eq}), such that $J(u)\leq c(R_1,R_2).$ This concludes the proof.
\end{proof}

\section{Existence of multiple sign-changing solutions}\label{proofsign}
 The mountain-pass principle for sign-changing solution, i.e. Theorem 3.7 in \cite{clapppacella}, holds in the Banach space $W^{1,p}_{0}(\Omega)^G.$ More precisely we have the following
\begin{lemma}\label{mmm}
Let $W$ be a finite dimensional subspace of $W_0^{1,p}(\Omega)^G$
and let $d := \sup_W J$. If $J$ satisfies $(PS)_c$ relative to $\Pi_0$ in $W_0^{1,p}(\Omega)^G$ (see Section \ref{GCsec}) for every $c \leq d$, then $J$ has at least $\textrm{dim}(W) - 1$
pairs of sign changing critical points $u \in W_0^{1,p}(\Omega)^G$ with $J (u) \leq d.$
\end{lemma}
\begin{proof}
In [Theorem 3.7, \cite{clapppacella}] use the pseudo-gradient flow of $J.$
\end{proof}

\begin{proof}[Proof of Theorem \ref{coronn2}]Fix $\delta<\varepsilon_0,$ where $\varepsilon_0$ is given by Corollary \ref{PSJc}. From Lemma \ref{cal} and Lemma \ref{homog} there exist $\{\omega\}_{i=1,...,l+1}\subset \mathcal N(\Omega)^G$ and $R_\delta>0$ such that if $R_1/R_2<R_\delta$ then $$\max_{W_k} J\leq \sum_{i=1}^{k+1}\max_{t>0}J(t\omega_i)\leq (k+1)c_\infty+\delta<(l+1)c_\infty+\varepsilon_0,$$ where $W_k:=\textrm{span}\,(\omega_1,...,\omega_{k+1})$ for $k=1,...,l.$ Notice that $\{\omega\}_{i=1,...,l+1}$ are linearly independent since they have disjoint supports.\\ From Corollary \ref{PSJc} and Lemma \ref{mmm} there exist $l$ couples of sign-changing critical points of $J,$ namely $\pm u_1,...,\pm u_l\in \mathcal N(\Omega)^G$ with $$J(u_k)\leq (k+1)c_\infty+\delta,\quad k=1,...,l,$$ and this concludes the proof.
\end{proof}
\begin{proof}[Proof of Theorem \ref{fixhol2}]
Fix $m\in \mathbb N$ and $0<R_1<R_2.$ Define $l_0:=c_\infty^{-1}(m+1)c(R_1^{\frac{1}{m+1}},R_2^{\frac{1}{m+1}}).$ From Lemma \ref{cal} and Lemma \ref{homog} there exist $\{\omega\}_{i=1,...,m+1}\subset \mathcal N(\Omega)^G$ such that $$\max_{W_k} J\leq \sum_{i=1}^{k+1}\max_{t>0}J(t\omega_i)\leq (k+1)c(R_1^{\frac{1}{m+1}},R_2^{\frac{1}{m+1}})\leq l_0 c_\infty,$$ where $W_k:=\textrm{span}\,(\omega_1,...,\omega_{k+1})$ for $k=1,...,m.$ \\ For every closed subgroup $G\subset O(N)$ with $l=l(G)>l_0$, from Corollary (\ref{PSJc}) and Lemma \ref{mmm} there exist $m$ couple of sign-changing critical point of $J,$ namely $\pm u_1,...,\pm u_m\in \mathcal N(\Omega)^G$ with $$J(u_k)\leq (k+1)c(R_1^{\frac{1}{m+1}},R_2^{\frac{1}{m+1}}),\quad k=1,...,m,$$ and this concludes the proof.
\end{proof}

 \section{Proof of Theorem \ref{GlobalPositive} and Theorem \ref{GlobalSign}}\label{proofglo}
\begin{proof} [Proof of Theorem \ref{GlobalPositive}] We follow the line given by M. Willem in [\cite{Will}, Theorem 8.13], inspired by \cite{BCoron}. The result is a consequence of the observations of M. Clapp in \cite{clapp} on the semilinear case and the global compactness in \cite{mercuriwillem} for the p-Laplacian operator. A sketch of the proof is due, and we give it in several steps. \newline
1)  Since the sequence $\{u_n\}_n$ is bounded in $W^{1,p}_{0}(\Omega),$ passing if necessary to a subsequence, we can assume that $u_n\rightharpoonup v_0$ in $W^{1,p}_0(\Omega)$ and $u_n\rg v_0$ a.e. on $\Om$.  By [Lemma 3.5, \cite{mercuriwillem}] and the symmetric criticality principle, it follows that $\phi'(v_0)=0$ and $u^1_n:=u_n-v_0$ is such that
\begin{align*}
&i) \,\,\|u^1_n\|^p=\|u_n\|^p-\|v_0\|^p+o(1),\\
&ii)\,\, \phi_\infty(u^1_n)\rightarrow c-\phi(v_0),\\
&iii) \,\,\phi'_\infty(u^1_n)\rightarrow 0 \quad \textrm{in}\,\, W^{-1,p'}(\Omega) .
\end{align*}
2) If $u^1_n\rightarrow 0$ in $L^{p^*}(\Omega),$ since $\phi'_\infty(u^1_n)\rightarrow 0$ in $W^{-1,p'}_{0}(\Omega),$ we have that $u^1_n\rightarrow 0$ in $W^{1,p}_0(\Omega)$ and the proof is complete. Otherwise we can assume that
$$\int_{\Omega}|u^1_n|^{p^*}dx>\delta$$ for some $0<\delta<(S/2)^{N/p}.$ Introducing the L\'evy concentration function $$L_n(r):=\sup_{y\in \R^N}\int_{B(y,r)}|u^1_n|^{p^*}dx, $$ since $L_n(0)=0$ and $L_n(\infty)>\delta,$ there exists a sequence $\{\lambda^1_n\}_n\subset ]0,\infty[$ and   a sequence $\{y^1_n\}_n\subset \Omega$ such that
\begin{equation*}
\delta=\sup_{y\in \R^N}\int_{B(y,\lambda^1_n)}|u^1_n|^{p^*}dx=\int_{B(y^1_n,\lambda^1_n)}|u^1_n|^{p^*}dx.
\end{equation*}\\
3) Given a closed subgroup $H\subset G$ define $$\textrm{Fix}\,(H):= \{x\in \R^N: gx=x,\, \forall g\in H\}$$ and for every $y\in \R^N$ $y^H$ to be the euclidian projection of $y$ onto $\rm{Fix}\,(H).$
In [\cite{clapp}, p.121] is proved the existence of a closed subgroup $\Gamma \subset G$ such that, passing if necessary to a subsequence: \\
a) $\# \Gamma <\infty,$\\
\\
b) $G_{y_n^\Gamma}=\Gamma$ for all $n,$ \\
\\
c) $(\lambda_n^1)^{-1}|g y_n^{1\,\Gamma}-g'y_n^{1\,\Gamma}|\rightarrow\infty$ as $n\rightarrow \infty$ for all $[g]\neq[g']\in G /\Gamma,$ \\
\\
d)$(\lambda_n^1)^{-1}|y_n^1-y_n^{1\,\Gamma}|<C<\infty.$ \\
\\
We define on $$\Omega_n:=\frac{1}{\lambda^1_n}(\Omega-y_n^{1\,\Gamma})$$ the $\Gamma-$invariant sequence $v^1_n(x):=(\lambda^1_n)^{(N-p)/p}u^1_n(\lambda^1_n x+y_n^{1\,\Gamma}).$ We can assume that $v^1_n\rightharpoonup v_1$ in $\mathcal D^{1,p}(\R^N)^\Gamma$ and $v^1_n\rightarrow v_1$ a.e. on $\R^N.$ From d) we observe that
\begin{equation}\label{cont}
\delta=\int_{B(y_n^{1\,\Gamma},\lambda^1_n(1+C))}|u^1_n|^{p^*}dx=\int_{B(0,1+C)}|v^1_n|^{p^*}dx.
\end{equation}
4) We prove that $v_1\neq 0.$ Let $f_n:=(f^1_n,...,f^N_n) \in (L^{p'}(\Omega))^{N}$ defined by the representation $$(\phi'_\infty(u^1_n),h)=\sum^N_{i=1}\int_{\Omega}f^i_n \partial_i h dx,\quad \quad \forall h\in W^{1,p}_0(\Omega).$$ Define $g_n:=(\lambda^1_n)^{(N-p)/p}f_n(\lambda^1_n x+y_n^{1\,\Gamma}).$ We have that $$(\phi'_\infty(v^1_n),h)=\sum^N_{i=1}\int_{\Omega_n}g^i_n \partial_i h dx,\quad \quad \forall h\in W^{1,p}_0(\Omega_n)$$ and, since $\phi'_\infty(u^1_n)\rightarrow 0,$
$$\sum^N_{i=1}\int_{\Omega_n}|g^i_n|^{p'}dx=\sum^N_{i=1}\int_{\Omega}|f^i_n|^{p'}dx=o(1).$$ \newline
We assume now, by contradiction, that $v_1=0.$ Then, passing to a subsequence, we can assume that $v^1_n\rightarrow 0$ in $L^p_{loc}(\R^N).$ Take $h\in \mathcal D(\R^N)$ such that  $\textrm{supp} \,h\subset B(y,1)$ for some $y\in\R^N.$ From the H\"older and Sobolev inequalities, it follows that
 $$ \int|h|^p|v^1_n|^{p^*}\leq S^{-1}\Big ( \int_{\textrm{supp}\,h} |v^1_n|^{p^*}\Big)^{p/N}\int |\nabla(h v^1_n)|^{p}. $$
 Hence, since $v^1_n\rightarrow 0$ in $L^p_{loc}(\R^N),$ we have

 \begin{eqnarray}
 \int_{\Omega_n}|\nabla (h v^1_n)|^p
 &=&\int|\nabla v^1_n|^{p-2}\nabla v^1_n \nabla (|h|^p v^1_n)+o(1) \nonumber\\
 &=& \int |h|^p |v^1_n|^{p^*}+\sum^N_{i=1}\int_{\Omega_n} g^i_n \partial_i(|h|^p v^1_n)+o(1) \nonumber\\
 &\leq&  S^{-1} \delta^{p/N} \int |\nabla (h v^1_n)|^p+o(1) \nonumber\\
 &\leq& \frac{1}{2}\int |\nabla (h v^1_n)|^p+o(1). \nonumber
 \end{eqnarray}
It follows that $\nabla v^1_n\rightarrow 0$ in $L^{p}_{loc}(\R^N)$ and by the Sobolev inequality we have that $v^1_n\rightarrow 0$ in $L^{p^*}_{loc}(\R^N).$ This is in contradiction with (\ref{cont}). Hence $v_1\neq 0.$ \newline
\noindent 5) Since $\Omega$ is bounded we may assume $y_n^{1\,\Gamma}\rightarrow y_0^{1\,\Gamma} \in \bar{\Omega}$ and $\lambda^1_n\rightarrow \lambda^1_0\geq 0.$
If   $\lambda^1_0>0$ then, as a consequence of the fact that $u^1_n\rightharpoonup 0$ in $W^{1,p}_0(\Omega),$ we have
$v^1_n \rightharpoonup 0$ in $\mathcal D^{1,p}(\R^N)^\Gamma$ and this is a contradiction.
We can also rule out the case $\lambda^1_n\rightarrow 0$ and $$\liminf_{n\rightarrow\infty} \frac{1}{\lambda^1_n}\textrm{dist}\,(y_n^{1\,\Gamma},\partial \Omega)< \infty.$$
Indeed, we would have $y_0^{1\,\Gamma}\in \partial \Omega$ and $v_1,$ which is nonnegative,
would satisfy
\begin{eqnarray}
-\Delta_p u =   u^{p^*-1} &\textrm{in}& H, \nonumber \\
u =0 & \textrm{on} & \partial H, \nonumber
\end{eqnarray}
where $H$ is a halfspace. But [Theorem 1.1, \cite{mercuriwillem}] implies $v_1\equiv 0.$ \newline
It follows that, for some subsequence, $$\frac{1}{\lambda^1_n}\textrm{dist}\,(y_n^{1\,\Gamma},\partial \Omega)\rightarrow \infty, \quad \lambda^1_n\rightarrow 0. $$
Moreover $y_0^{1\,\Gamma}\in \Omega.$ By Step 1 and Lemma 3.6 in \cite{mercuriwillem} we have that $\phi'_\infty(v_1)=0$ and  by the strong maximum principle $v_1>0$.\\
6) Arguing as in [\cite{clapp}, p.124] pick a radial cut-off function $\chi\in \mathcal D(\R^N),\, 0\leq\chi\leq 1$ such that $\chi\equiv 1$ on $B(0,1)$ and $\chi\equiv 0$ on $\R^N\setminus B(0,2).$ Define also $$4\rho_n:=\min \{|g y_n^{1\,\Gamma}-g'y_n^{1\,\Gamma}|\,:\,[g]\neq[g']\in G /\Gamma\}.$$ From Step 3 and from property c) together with Lemma 3.6 in \cite{mercuriwillem} the sequence

$$u^2_n(x):=u^1_n(x)-\sum_{[g]\in G/\Gamma}(\lambda^1_n)^{(p-N)/p}v_1(g^{-1}\lambda_n^{-1}(x-g y_n^{1\,\Gamma})/\lambda^1_n)\chi(\rho_n^{-1}(x-gy_n^{1\,\Gamma}))\in W^{1,p}_0(\Omega)^G$$ satisfies
\begin{align*}
&i) \,\,\|u^2_n\|^p =\|u_n\|^p-\|v_0\|^p-|G/\Gamma|\|v_1\|^p+o(1),\\
&ii)\,\, \phi_\infty(u^2_n)\rightarrow c-\phi(v_0)-|G/\Gamma|\phi_\infty(v_1),\\
&iii) \,\,\phi'_\infty(u^2_n)\rightarrow 0 \quad \textrm{in}\,\, (W^{1,p}_{0}(\Omega))'.
\end{align*}

\noindent 7)  Since for any nontrivial critical point $u$ of $\phi_\infty$ we have $$S\|u\|^{p}_{L^{p^*}(\R^N)}\leq \|u\|^p=\|u\|^{p^*}_{L^{p^*}(\R^N)},$$ we get $$\phi_\infty(u)\geq c^*:= \frac{S^{N/p}}{N}.$$ As a consequence, the above procedure can be iterated only for a finite number of steps.

\end{proof}
\begin{proof}[Proof of Theorem \ref{GlobalSign}]
The proof requires small changes with respect to the proof of Theorem \ref{GlobalPositive}, taking into account that the nonexistence of sign-changing solutions in the halfspace is not known. Therefore we leave it out.
\end{proof}
\section{An extension}\label{ext}
In the proof of the above existence results it is remarkable the role played by the positive solutions constructed in a hierarchy of annular domains, and  belonging to the Nehari manifold. The results we obtained can be furthermore extended, in order to enlarge the class of domains $\Omega$ having a finite number of symmetries and nontrivial topology. Indeed, as observed in \cite{clappfaya}, Theorem \ref{fixhol} and Theorem \ref{fixhol2} require $l$ large enough, and this could not be the case in odd dimensions.  Following \cite{clappfaya}, we denote by $\Gamma$ a closed subgroup of $O(N)$ and by $D\subset \R^N$ a $\Gamma$-invariant smooth domain such that every orbit $\Gamma x$ is infinite on $D.$ We have the following
\begin{theorem}
Let $1<p<N.$  Then, there exists an increasing sequence $\{l_m\}$ of positive real numbers, depending on $p, \Gamma, D,$ such that if $D\subset\Omega,$ and if $\Omega$ is $G$-invariant, being $G$ a closed subgroup of $\Gamma$ such that $$\min_{x\in \Omega}\#Gx>l_m, $$
then (\ref{main eq}) possesses $m$ pairs of $G$-symmetric solutions $\pm u_1,..., \pm u_m$ with $u_1$ positive and $u_2,...,u_m$ sign-changing, and $$J(u_k)\leq l_k\frac{S^{N/p}}{N}, \quad k=1,...,m.$$
\end{theorem}
\begin{proof}
The proof follows with obvious modifications from \cite{clappfaya}, taking into account the compactness results of Section \ref{GCsec} and Lemma \ref{mmm}.
\end{proof}
\begin{rem}
Theorem \ref{fixhol} and Theorem \ref{fixhol2} correspond to the case $D=A_{R_1,R_2}$ and $\Gamma=O(N).$
\end{rem}

\subsection*{Acknowledgments}
  C. M. would like to thank the Department of Mathematics "Guido Castelnuovo" (Universit\'a di Roma "Sapienza") for the kind hospitality when this paper has been started.

%\begin{thebibliography}{10}
%
%\bibitem{smale} F. Cucker and S. Smale,
% \newblock {\em On the mathematical foundations of learning}.
%\newblock Bull. Amer. Math. Soc. (N.S.) {\bf 39} (2002), no. 1, 1�-49
%
%
%
%\end{thebibliography}

\end{document}